\documentclass{amsart}
\usepackage[english]{babel}
\usepackage[a4paper,top=2cm,bottom=2cm,left=3cm,right=3cm,marginparwidth=1.75cm]{geometry}

\usepackage{biblatex}
\renewbibmacro{in:}{\addcomma\addspace}

\usepackage{amsaddr,amsmath, amsfonts, amssymb,graphicx, tikz-cd,tikz,enumitem,mathtools,hyperref, mathabx, mathrsfs, mathabx, cleveref, mdframed, lipsum,csquotes}
\usetikzlibrary{calc,shapes}

\theoremstyle{plain}
\newtheorem{theorem}{Theorem}
\newtheorem{lemma}[theorem]{Lemma}
\newtheorem{corollary}[theorem]{Corollary}
\newtheorem{proposition}[theorem]{Proposition}

\theoremstyle{definition}

\newtheorem*{theorem*}{Theorem}
\newtheorem*{corollary*}{Corollary}
\newtheorem{example}[theorem]{Example}

\theoremstyle{remark}
\newtheorem{remark}[theorem]{Remark}

\newcommand{\lreg}[1]{\prescript{}{#1}{#1}}
\newcommand{\soc}[1]{\mathrm{soc}(#1)}
\newcommand{\fac}{^{\text{fct}}}
\newcommand{\img}{^{\text{im}}}

\begin{document}

\author[Buhphang et al.]{Ardeline M. Buhphang $^{\dagger}$, Rishabh Goswami $^{\dagger}$ and Amit Kuber $^{\ddagger}$}
\address{\centering$^{\dagger}$ Department of Mathematics, North-Eastern Hill University, Shillong, India\\
       $^{\ddagger}$ Department of Mathematics and Statistics, Indian Institute of Technology, Kanpur, India\\
       Corresponding Author: Amit Kuber}

\email{ardeline@nehu.ac.in, rishabhgoswami.math@gmail.com, askuber@iitk.ac.in}

\title{A note on injectivity of monomial algebras}

\begin{abstract}
We show that a monomial algebra $\Lambda$ over an algebraically closed field $K$ is self-injective if and only if each map $\mathrm{soc}(\prescript{}{\Lambda}{\Lambda})\to\prescript{}{\Lambda}{\Lambda}$ can be extended to an endomorphism of $\prescript{}{\Lambda}{\Lambda}$, and provide a complete classification of such algebras. As a consequence, we show that the class of self-injective monomial algebras is a subclass of Nakayama algebras.
\end{abstract}
\maketitle

\section{Introduction}
Let $K$ be an algebraically closed field and $\Lambda$ be a finite dimensional, associative, unital $K$-algebra. We will only consider finite dimensional left $\Lambda$-modules. The algebra $\Lambda$ is assumed to be a basic connected algebra, i.e., $\Lambda$ is (isomorphic to) a bound quiver algebra $KQ/\langle\rho\rangle$ for a finite quiver $Q$ and a set of relations $\rho$, thanks to \cite[Theorem~II.3.7] {assem2006elements}. Say that  $\Lambda$ is a \emph{monomial algebra} if it is (isomorphic to) a bound quiver algebra $ KQ/\langle\rho\rangle$, where $\rho$ consists only of monomials (paths). Monomial algebras are also known as \emph{zero relation algebras}. Tacitly, we assume that no two monomial relations in $\rho$ are comparable, i.e., no path in $\rho$ is properly contained in another path in $\rho$.

Recall that $\Lambda$ is left \emph{self-injective} if the left regular module $\lreg{\Lambda}$ is injective. A closely related notion of \emph{quasi-Frobenius algebras} was introduced and studied by Nakayama in \cite{nakayama1939frobeniusean,nakayama1941frobeniusean}; in fact, these two notions coincide in the case of finite dimensional algebras.

Self-injectivity is a strong condition, and one of the weaker conditions, socle-injectivity, was introduced by Amin et al. \cite{aminsoc}. Say that $\Lambda$ is \emph{socle-injective} if each map $\soc{\lreg{\Lambda}}\to\lreg{\Lambda}$ can be extended to an endomorphism of $\lreg{\Lambda}$. In general, socle-injectivity is strictly weaker than self-injectivity as the following example demonstrates. 
\begin{example}
    The polynomial algebra $K[x]$ is an infinite dimensional monomial $K$-algebra presented as the path algebra of the single loop quiver. It is socle-injective because $\soc{K[x]}=0$ but not self-injective for the map $xf(x)\mapsto f(x):xK[x]\to K[x]$ cannot be extended to an endomorphism of $K[x]$. 
\end{example}
In stark contrast to this example, our main theorem below establishes an equivalence between these two notions of injectivity in the context of finite dimensional monomial algebras.

\begin{theorem}\label{main}
The following conditions are equivalent for a finite dimensional monomial $K$-algebra $\Lambda=KQ/\langle\rho\rangle$:
\begin{enumerate}
    \item the algebra $\Lambda$ is self-injective;
    \item the socle $\soc{P(S)}$ is simple for each simple $S$, where $P(S)$ is the projective cover of $S$, and the map that sends the iso-class $[S]$ of a simple $\Lambda$-module $S$ to the iso-class $[\soc{P(S)}]$ is a permutation of the finite set of iso-classes of simple modules;
    \item the algebra $\Lambda$ is socle-injective;
    \item either $\Lambda\cong K$ or there are $n,l>1$ such that $Q=C_n$ and $\langle\rho\rangle=R_{C_n}^l$, where $C_n$ is the cyclic quiver with $n$ vertices and $R_{C_n}$ is the arrow ideal of $KC_n$.
\end{enumerate}
\end{theorem}

It is immediate from condition $(4)$ of the theorem that if $\Lambda\not\cong K$ then all indecomposable finitely generated projective as well as injective $\Lambda$-modules are uniserial with Loewy length $l\mbox{-}1$. Hence from \cite[\S~V.3]{assem2006elements} we get the following.

\begin{corollary}
A self-injective finite dimensional monomial $K$-algebra is a Nakayama algebra.
\end{corollary}

The permutation in condition $(2)$ of Theorem \ref{main} is known as the Nakayama permutation. The equivalence $(1)\Leftrightarrow(2)$ is true for all finite dimensional algebras over arbitrary fields, and the reader is referred to lecture notes of Farnsteiner \cite{Farnsteiner2005SELFINJECTIVEAT} for a proof, which essentially follows from \cite[\S~I.4,II.4]{auslander1995representation}. While the implication $(1)\Rightarrow(3)$ is obvious from the definition, the proof of the implication $(4)\Rightarrow(1)$ follows from \cite[Proposition~V.3.8]{assem2006elements}. To complete the proof of the theorem we will show that $(3)\Rightarrow(2)\Rightarrow(4)$.

\section{Preliminaries}
Throughout the rest of the paper, assume that $\Lambda=KQ/\langle\rho\rangle$ is a finite dimensional monomial algebra. Denote the admissible ideal $\langle\rho\rangle$ by $I$ for simplicity of notation. We will freely use the equivalence between $KQ/I$-modules and bound $K$-representations of $(Q,I)$. A $K$-representation $((V_j)_{j\in Q_0}, (\varphi_\gamma)_{\gamma\in Q_1})$ of $(Q,I)$ will be abbreviated as $(V_j,\varphi_\gamma)$ when the quiver $Q$ is clear from the context. The reader is referred to the standard text \cite{assem2006elements} for unexplained terminology, notations and results about bound quiver algebras, except for the convention regarding paths: we say that $\alpha_n\cdots\alpha_1$ is a path when the source of $\alpha_{i+1}$ is the same as the target of $\alpha_i$ for $1 \le i < n$. 

Let $Q_0:=\{1,2,\cdots, m\}$ and $e_i$ denote the stationary path associated to vertex $i$. Recall from \cite[Lemma~II.2.4]{assem2006elements} that the set $\left\{\bar e_i \mid i \in Q_0\right\}$ is a complete set of primitive orthogonal idempotents of the algebra $KQ/I$, where $\bar e_i:=e_i+I$. Let $S(i)$ denote the simple module associated to vertex $i$.

Recall from \cite[Lemma~I.5.3]{assem2006elements} that the left regular module $\lreg{\Lambda}$ can be decomposed as $\bigoplus_{i\in Q_0}P(i)$, where $P(i)$ is the projective indecomposable module isomorphic to $\Lambda \bar e_i$, where $\bar e_i:=e_i+I$. The following result describes the projective indecomposable modules $P(i)$.
\begin{lemma}\cite[Lemma~III.2.4]{assem2006elements}\label{lem:projrep}
If $(P(i)_j, \varphi_\gamma)$ is the representation of $(Q,I)$ corresponding to $P(i)$, then $P(i)_j$ is the $K$-vector space with basis the set of all the $\bar{p} := p + I$, with $p$ a path from $i$ to $j$; whereas, for an arrow $\gamma: j \rightarrow k$, the $K$-linear map $\varphi_\gamma: P(i)_j \rightarrow P(i)_k$ is defined as the left multiplication by $\overline{\gamma}:= \gamma + I.$ 
\end{lemma}

Complete classification of indecomposable finitely generated modules for a monomial algebra is not yet known. However, Crawley-Boevey \cite{crawleytree} studied a subclass of indecomposable finitely generated modules, known as tree modules, for monomial algebras, and found a basis of Hom-set between two tree modules. We briefly recall these definitions and results.

A \emph{tree} is a simply connected finite quiver $T=(T_0,T_1)$. A \emph{rooted tree} is a tree $T$ with a special vertex $\ast$ such that there is a path from $\ast$ to every other vertex. There are two special types of subtrees of a tree. Say that a subtree $T\fac$ is a \emph{factor subtree} of $T$ if for each arrow $v\xrightarrow{a}w$ in $T$, whenever $w\in T\fac$ then $v\in T\fac$. Dually, say that a subtree $T\img$ is an \emph{image subtree} of $T$ if for each arrow $v\xrightarrow{a}w$ in $T$, whenever $v\in T\img$ then $w\in T\img$.

\begin{remark}\label{rmk:facimgrootedtree}
Suppose $T$ is a rooted tree with root $\ast$.
\begin{enumerate}
    \item If $T\fac$ is a factor subtree of $T$ then $\ast\in T\fac$.
    \item If $T\img$ is an image subtree of $T$ then $\ast\in T\img$ if and only if $T=T\img$.
    \item If $T\img$ is an image subtree of $T$ then $T\img$ is an induced rooted subtree of $T$.
\end{enumerate}
\end{remark}

For each tree $T$, we denote by $V_T$ the $KT$-module corresponding to the $K$-representation $((K)_{y\in T_0}, (1_K)_{\gamma\in T_1})$. If $ F:T\to Q$ is a quiver morphism satisfying $ F(a)\neq F(b)$ whenever $a,b\in T_1$ are distinct but share either source or target, and $ F(p)\notin I$ for any path $p$ in $T$ then there is a push-down functor $ F_\lambda:KT\mbox{-}\mathrm{Mod}\to KQ/I\mbox{-}\mathrm{Mod}$ defined by $$( F_\lambda (U_y,\psi_\gamma))_j := \bigoplus_{y' \in  F ^{-1}(j)}U_{y'},\quad( F_\lambda (U_y,\psi_\gamma))_\alpha := \bigoplus_{\gamma' \in  F ^{-1}(\alpha)}\psi_{\gamma'},$$ for a $KT$-representation $(U_y,\psi_\gamma)$, $j\in Q_0,\ \alpha\in Q_1$. The $KQ/I$-module $ F_\lambda V_T$ is called a \emph{tree module}. Gabriel showed in \cite[\S~3.5 and \S~4.1]{gabrieluniversal} that tree modules are indecomposable.

For tree modules $X,Y$ given by the data $(T_X, F_X)$ and $(T_Y, F_Y)$ respectively, a \emph{graph map} from $X$ to $Y$ is a triple $(T_X\fac,T_Y\img,\sigma:T_X\fac\to T_Y\img)$, where $T_X\fac$ is a factor subtree of $T_X$, $T_Y\img$ is an image subtree of $T_Y$, and $\sigma$ is an isomorphism of trees such that $ F_Y\circ\sigma= F_X$.
\begin{theorem}\cite{crawleytree}\label{thm:crawboe}
If $X$ and $Y$ are tree modules for $\Lambda$ given by the data $(T_X, F_X)$ and $(T_Y, F_Y)$ respectively then $\mathrm{Hom}_\Lambda(X,Y)$ is a finite dimensional vector space with a basis given by graph maps $(T_X\fac,T_Y\img,\sigma:T_X\fac\to T_Y\img)$.
\end{theorem}

The next result may be known to experts but we could not find it in the literature.
\begin{proposition}\label{projaretree}
The modules $S(i)$ and $P(i)$ are tree modules for each $i\in Q_0$.
\end{proposition}

\begin{proof}
The result is trivial for simple modules, and thus we only prove it for projective modules.

Given $i\in Q_0$, we construct a rooted tree $T_i$ together with a quiver morphism $ F_i:T_i\to Q$ as the following algorithm describes.
\begin{enumerate}
    \item[Step 1] Add a vertex $\star_i$ as the root and set $ F_i(\ast_i):=i$. 
    \item[Step 2] Suppose $v$ is a newly added vertex and $p$ is the unique path from $\ast_i$ to $v$ (possibly of length $0$). If $\beta$ is an arrow in $Q$ with source $ F_i(v)$ such that the concatenation $\beta F_i(p)$ does not lie in $I$ then add a new vertex $w$ with an arrow $v\xrightarrow{b}w$ and set $ F_i(b):=\beta$.
    \item[Step 3] Repeat Step 2 until paths in the image of $ F_i$ cannot be extended further.
\end{enumerate}
Since there is an absolute bound on the length of paths in $(Q, \rho)$, the above algorithm terminates when the image of $T_i$ under $ F_i$ has covered all maximal paths in $(Q, \rho)$ with source $i$. Finally, using Lemma \ref{lem:projrep} it is straightforward to verify that the $( F_i)_\lambda V_{T_i}\cong P(i)$, so as to conclude that $P(i)$ is indeed a tree module. 
\end{proof}

\begin{remark}\label{rmk: distinctpathsinTj}
The rooted tree $T_j$ constructed in the proof above has the following property: if $p_1$ and $p_2$ are two maximal paths with source at the same vertex then $ F_i(p_1)\neq F_i(p_2)$. As a consequence, there is no automorphism $\psi$ of $T_i$ other than identity such that $F_i\circ\psi=F_i$.
\end{remark}
The following example illustrates the algorithm in the proof of Proposition \ref{projaretree}.

\begin{example}
    Consider the algebra $KQ'/\langle\rho'\rangle$, where $(Q', \rho')$ is as shown in the left hand side diagram of Figure \ref{Fig:1}. The data $(T_1, F_1)$ associated with the projective indecomposable module $P(1)$ is shown on the right hand side in the same figure, where the value under $F_1$ of the first coordinate is shown in the second coordinate. \begin{figure}[h]\centering

\[\begin{tikzcd}
	1 & 2 & 3 & {(y,3)} & {(v,2)} & {(\ast_1 ,1)} & {(w,2)} \\
	&&&& {(x, 3)} && {(z, 3)}
	\arrow["{(d, \delta)}"', from=1-5, to=1-4]
	\arrow["{(a, \alpha)}"', from=1-6, to=1-5]
	\arrow["{(c_1, \gamma)}", from=1-5, to=2-5]
	\arrow["{(c_2, \gamma)}"', from=1-7, to=2-7]
	\arrow["{(b, \beta)}", from=1-6, to=1-7]
	\arrow["\alpha", bend left, from=1-1, to=1-2]
	\arrow["\delta"',bend right, from=1-2, to=1-3]
	\arrow["\beta"', bend right, from=1-1, to=1-2]
	\arrow["\gamma", bend left, from=1-2, to=1-3]
\end{tikzcd}\]
\caption{A quiver $Q'$ with $\rho' = \{\delta\beta\}$ (left), the data $(T_1, F_1)$ (right) }
\label{Fig:1}
\end{figure}
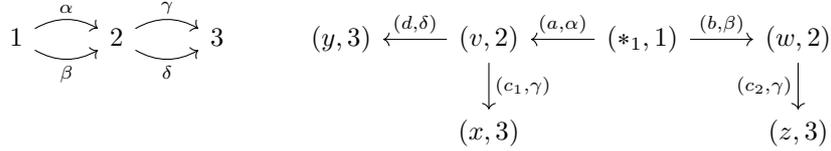
\end{example}

Since socle commutes with direct sums, we get that $\soc{\lreg{\Lambda}}=\bigoplus_{i\in Q_0}\soc{P(i)}$. It is easy to compute $\soc{P(i)}$ using Lemma \ref{lem:projrep} and \cite[Lemma~III.2.2]{assem2006elements}. Furthermore, using Proposition \ref{projaretree} it is possible to relate this computation with the rooted tree $T_i$.
\begin{corollary}\label{cor:socProj}
For $i\in Q_0$, the only simple summands of $\soc{P(i)}$ correspond to $S(j)$, where $j$ is the target of a maximal path in $(Q, I)$ with source $i$ that does not contain a subpath in $I$. In other words, using the notations of the proof of Proposition \ref{projaretree}, we get $$\soc{P(i)}\cong\bigoplus_{\mbox{x is a leaf in $T_i$}}S(F_i(x)).$$
\end{corollary}

\section{The proof of $(2)\Rightarrow(4)$ of Theorem \ref{main}}
Suppose $\Lambda=KQ/I$ satisfies condition $(2)$. We further assume that $\Lambda\not\cong K$, which is equivalent to $|Q_0|=m>1$. 

Since connectedness of $\Lambda$ is equivalent to the connectedness of $Q$ by \cite[Lemma II.2.5]{assem2006elements}, in order to show that $Q=C_m$, it is sufficient to show that the outdegree of each vertex in $Q$ is exactly $1$.

If the outdegree of a vertex $i$ exceeds $1$ then it follows from Corollary \ref{cor:socProj} that $\soc{P(i)}$ is a direct sum of two or more simple modules, a contradiction to condition $(2)$. Thus the outdegree of all vertices should be at most 1. Moreover, since the quiver is connected at most one vertex in $Q$ can have outdegree $0$. Since $m\geq 1$, we conclude that $Q$ is either a path or a cycle. We will rule out the possibility of a path.

Assume that $Q$ is a path, and that the outdegree of $i_0\in Q_0$ is $0$. Let $i_1\xrightarrow{\alpha}i_0$ be the unique arrow with target $i_0$. Since the admissible ideal $I$ contains paths of length at least $2$, Corollary \ref{cor:socProj} yields that $\soc{P(i_1)}\cong S(i_0) \cong\soc{P(i_0)}$, again a contradiction to injectivity of Nakayama permutation guaranteed by $(2)$. Thus $Q=C_m$. Without loss, we may assume that $Q_1=\{i\xrightarrow{\alpha_i}(i+_m1)\mid i\in Q_0\}$, where $+_m$ is addition modulo $m$ with values lying in $Q_0=\{1,2,\cdots,m\}$.

Now we show that $|\rho|=m$. Clearly $|\rho|>0$ for otherwise $KC_m/\langle\rho\rangle=KC_m$ is an infinite dimensional algebra. On the one hand, if $|\rho|<m$ then there is a vertex, say $i_0$, that does not appear as the source of any path in $\rho$ but $(i_0+_m1)$ does. Then Corollary \ref{cor:socProj} yields that $\soc{P(i_0+_m1)}\cong\soc{P(i_0)}$, yet again a contradiction. On the one hand, if $|\rho|>m$ then by the pigeonhole principle at least two paths in $\rho$ have the same source, which is a contradiction to the assumption on $\rho$ that no two paths are comparable. In fact, we showed that no two distinct paths in $\rho$ have the same source. Let $p_i$ denote the unique path in $\rho$ with source $i$.

Finally, we argue that any two relations in $\rho$ have the same length. Since no two distinct paths in $\rho$ are comparable, we see that $|p_i|\leq|p_{(i+_m1)}|$ for each $i\in Q_0$. Since $Q$ is a cycle, we conclude that $|p_i|=|p_{(i+_m1)}|$ for each $i$. Let $l$ denote the common length of all such paths. Since $\langle\rho\rangle$ is admissible, we conclude that $l>1$ and $\langle\rho\rangle=R_{C_m}^l$.

\section{The proof of $(3)\Rightarrow(2)$ of Theorem \ref{main}}
Suppose $\Lambda=KQ/I$ is socle-injective.  Then for each map $f$ in the left diagram of Figure \ref{fig:soc-inj}, there exists some $g$ such that the diagram commutes. Since $\lreg{\Lambda}=\bigoplus_{i\in Q_0}P(i)$ and socle commutes with direct sums, the inclusion $\soc{\lreg{\Lambda}}\hookrightarrow\lreg{\Lambda}$ can be expressed as the direct sum of the individual inclusions $\soc{P(i)}\hookrightarrow P(i)$.

\begin{figure}[h]
    \centering
    \begin{minipage}{0.45\textwidth}
\[\begin{tikzcd}
	{\mathrm{soc}(\lreg{\Lambda})} && {\lreg{\Lambda}} \\
	& {\lreg{\Lambda}}
	\arrow["{\forall f}"', from=1-1, to=2-2]
	\arrow[hook, from=1-1, to=1-3]
	\arrow["{\exists g}", dotted, from=1-3, to=2-2]
\end{tikzcd}\]    
    \end{minipage}
    \begin{minipage}{0.45\textwidth}
\[\begin{tikzcd}
	{S(i)} && {P(j)} \\
	& {P(q)}
	\arrow["{\forall}"',hook, from=1-1, to=2-2]
	\arrow["{\forall}",hook, from=1-1, to=1-3]
	\arrow["{\exists}",dotted, from=1-3, to=2-2]
\end{tikzcd}\]    
    \end{minipage}
    \caption{Socle-injectivity of $\Lambda$}
    \label{fig:soc-inj}
\end{figure}
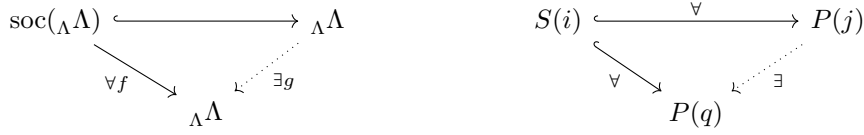

First note that $\mathrm{Hom}_\Lambda(\soc{\lreg{\Lambda}},\lreg{\Lambda})\cong\bigoplus_{i\in Q_0}\bigoplus_{j\in Q_0}\mathrm{Hom}_\Lambda(\soc{P(i)},P(j))$. Moreover, in view of Corollary \ref{cor:socProj}, it is sufficient to compute a basis $\mathrm{Hom}_\Lambda(S(i),P(j))$ for $i,j\in Q_0$. Since the simple module $S(i)$ is a tree module given by the data $(\{\ast\},\ast\mapsto i)$, in view of Theorem \ref{thm:crawboe}, the set of graph maps is in bijective correspondence with the set of image subtrees of $(T_j,F_j)$ isomorphic to $(\{\ast\},\ast\mapsto i)$, which the set of leaves $v$ in $T_j$ satisfying $F_j(v)=i$ in view of Remark \ref{rmk:facimgrootedtree}(3). Thus the dimension of the space $\mathrm{Hom}_\Lambda(S(i),P(j))$ is equal to the multiplicity of $S(i)$ as a summand in $\soc{P(j)}$. As a result, the left hand side diagram of Figure \ref{fig:soc-inj} can be understood in terms of several diagrams of the type described on the right hand side of the same figure.

Consider the right side diagram in Figure \ref{fig:soc-inj}, i.e., $S(i)$ occurs as a direct summand of $\soc{P(j)}$ and $\soc{P(q)}$. Choose maximal paths $p_1$ and $p_2$ in $T_j$ and $T_q$ respectively corresponding to these simple summands of the socles. Then socle-injectivity ensures the existence of a graph map $(T_j\fac,T_q\img,\sigma:T_j\fac\to T_q\img)$  from $P(j)$ to $P(q)$ that maps $p_1$ to $p_2$.

Suppose $j\neq q$. Since $\ast_j\in T_j\fac$ by Remark \ref{rmk:facimgrootedtree}(1) and $F_q(\sigma(\ast_j))=F_j(\ast_j)=j$, we get that $\sigma(\ast_j)\neq\ast_q$. Let $p$ be the unique positive length path in $T_q$ from $\ast_q$ to $\sigma(\ast_j)$. Since the isomorphism $\sigma$ preserves lengths, we see that $|p_2|=|p|+|\sigma(p_1)|=|p|+|p_1|$, and hence $|p_1|<|p_2|$. Since the roles of $P(j)$ and $P(q)$ can be reversed in the right side diagram in Figure \ref{fig:soc-inj}, we can also obtain $|p_1|>|p_2|$ using a similar argument, which is absurd. Therefore, socles of distinct projective indecomposable modules do not share a simple summand.

Now suppose $j=q$ but $p_1\neq p_2$. If $\sigma(\ast_j)=\ast_j$ then $\ast_j\in T_j\img$, and hence $T_j=T_j\img$ by Remark \ref{rmk:facimgrootedtree}(2). Since $p_1\neq p_2$ but $\sigma(p_1)=p_2$, we observe that $\sigma$ is a non-trivial automorphism of $T_j$, which is a contradiction to Remark \ref{rmk: distinctpathsinTj}. Therefore, $\sigma(\ast_j) \neq\ast_j$. Now using an argument similar to the above paragraph we conclude that the multiplicity of a simple summand of the socle of an projective indecomposable module is at most $1$.

The conclusions of the above two paragraphs guarantee the existence of a Nakayama permutation, thereby proving $(2)$.

\section*{Declarations}
\subsection*{Ethical approval}
Not applicable

\subsection*{Competing interests}
Not applicable

\subsection*{Authors' contribution}
All authors have contributed equally to all sections.

\subsection*{Funding}
Not applicable
\subsection*{Availability of data and materials}
Not applicable

\printbibliography
\end{document}